\documentclass[12pt,oneside]{amsart}

\usepackage{amssymb}
\usepackage[margin=1in]{geometry}  
\usepackage{enumerate}

\usepackage{ulem}      
\usepackage{amsmath, amsthm, amssymb}
\usepackage{xspace}
\usepackage{mathrsfs}
 \usepackage{graphicx}  
 \usepackage{color}	
 \usepackage{mathabx}
\newtheorem{theorem}{Theorem}[section]

\newtheorem{lemma}[theorem]{Lemma}

\newtheorem{example}[theorem]{Example}

\newcommand{\scr}{\mathscr}   
   
      \makeatletter
      \def\@setcopyright{}
      \def\serieslogo@{}
      \makeatother

\begin{document}

   \author{Amin Bahmanian}
   \address{Department of Mathematics and Statistics, 221 Parker Hall\\
 Auburn University, Auburn, AL USA   36849-5310}
   \email{mzb0004@auburn.edu}

   \author{C. A. Rodger}
   \address{Department of Mathematics and Statistics, 221 Parker Hall\\
 Auburn University, Auburn, AL USA   36849-5310}
   \email{rodgec1@auburn.edu}

   \title[ Embedding factorizations for 3-uniform hypergraphs]{ Embedding factorizations for 3-uniform hypergraphs}

   \begin{abstract}
In this paper, two results are obtained on a hypergraph embedding problem. The proof technique is itself of interest, being the first time amalgamations have been used to address the embedding of hypergraphs.

The first result finds necessary and sufficient conditions for the embedding a hyperedge-colored copy of the complete 3-uniform hypergraph of order $m$, $K_m^3$, into an $r$-factorization of  $K_n^3$, providing that $n> 2m+(-1+\sqrt{8m^2-16m-7})/2$. 

The second result finds necessary and sufficient conditions for an embedding when not only are the colors of the hyperedges of $K_m^3$ given, but also the colors of all the ``pieces" of hyperedges on these $m$ vertices are prescribed (the ``pieces" of hyperedges are eventually extended to hyperedges of size 3 in $K_n^3$ by adding new vertices to the hyperedges of size 1 and 2 during the embedding process).     

Both these results make progress towards settling an old question of Cameron on completing partial 1-factorizations of hypergraphs.

   \end{abstract}

   \keywords{Keyword. amalgamations, detachments, complete 3-uniform hypergraphs, embedding, factorization, decomposition}

   \date{\today}

   \maketitle

\section {Introduction}

A \textit{k-hyperedge-coloring} of a hypergraph $\scr G$ (we give a  precise definition of a hypergraph in the next section) is a mapping $K:E(\scr G)\rightarrow C$, where $C$ is a set of $k$ \textit{colors} (often we use $C=\{1,\ldots,k\}$), and the hyperedges of one color are said to form a \textit{color class}.  Let $\scr G$ be a hypergraph, and let $\scr H$
be a family of hypergraphs. We say that $\scr G$ has an $\scr H$-decomposition if there exists a partition $\{E(\scr H_{1}),\ldots,E(\scr H_{m})\}$ of $E(\scr G)$ such that $\scr H_{i}$ is isomorphic to a hypergraph in $\scr H$ for $1\leq i\leq m$.

The general setting for this paper is as follows. Let $\scr H$ and  $\scr H^*$ be two families of hypergraphs. Given a hypergraph $\scr G$ with an $\scr H$-decomposition and a hypergraph $\scr G^*$ which is a super-hypergraph of $\scr G$, under what circumstances can one extend the $\scr H$-decomposition of $\scr G$ into an $\scr H^*$-decomposition of $\scr G^*$? 
In other words, given a hyperedge-coloring of $\scr G$ in which each color class induces a hypergraph in $\scr H$, is it possible to extend this coloring to a hyperedge-coloring of $\scr G^*$ so that each color class of $\scr G^*$ induces a hypergraph in $\scr H^*$? Most naturally, $\scr G$ is usually taken to be the complete $h$-uniform hypergraph on $m$ vertices, $K_m^h$. 

Solving this problem requires knowledge about   hypergraph decompositions; compared to graph decompositions, very little is known about these,  even for special cases.
 Perhaps the best evidence for this difficulty is the long standing open problem of Sylvester in 1850 (in connection with Kirkman's famous Fifteen Schoolgirls Problem \cite{Kirk1847}) which asks whether it is possible to find a $1$-factorization of $K_n^h$ (see the next section for definitions). It took 120 years before Baranyai finally  settled this conjecture \cite{Baran75}. 
After Baranyai's proof appeared, in 1976 Cameron \cite{PJCameron76} asked the following question:

\begin{itemize}
\item [] Under what conditions can partial 1-factorizations of $K_m^h$ be extended to 1-factorizations of $K_n^h$?
\end{itemize}
\noindent This problem is wide open and to the authors$'$ best knowledge, the only partial results address the very special case of embedding  a $1$-factorization of $K_m^h$ into a  $1$-factorization of $K_n^h$ \cite{BaranBrouwer77, HaggHell93}.

Here we make some progress toward settling this problem, considering the  following related general embedding problem that is natural in their own right. When can a   hyperedge-coloring of a given hypergraph $\scr G$ on $m$ vertices be embedded into a hyperedge-coloring of $K_n^3$ in such a way that each color class forms an $r$-factor? So the special case when $r = 1$ and $\scr G=K_m^h$ addresses the Cameron question in the situation where the given partial 1-factors are all defined on a set of $m$ vertices.

 In Section \ref{embedfactorizationcor1},  we assume that precisely the hyperedges of size 3 on $m$ vertices have been colored; that is, the given hypergraph is $\scr G=K_m^3$), giving a complete solution if $n> 2m+(-1+\sqrt{8m^2-16m-7})/2$ (see Theorem \ref{facembedgen1}).  
Lemma \ref{lowerboundmn} then shows that Theorem \ref{facembedgen1} is not true if this bound on $n$ is replaced by $n \geq 2m-1$. In Section \ref{embedrestrfac} we assume that not only the hyperedges of size 3 are colored, but so are all the ``pieces" of hyperedges of $K_n^3$ that contain one or two of the given $m$ vertices (i.e. $n-m$ and $\binom{n-m}{2}$ copies of the hyperedges in $K_m^2$ and $K_m^1$, respectively); these pieces are built up to hyperedges of size 3 when the new vertices are added. In this case the problem is completely solved in Section \ref{embedrestrfac}, providing necessary and sufficient conditions (see Theorem \ref{restricembedth1}).   

The results in this paper supplement embedding results for graphs. Such results typically take a given edge-coloring of all the edges of a smaller complete graph and extend it to an edge-coloring of all the edges of a bigger complete graph in such a way that each color class is one of a given family of graphs. Hilton \cite{H2} used amalgamations to solve the  problem of embedding an edge-coloring of $K_m$ into a  Hamiltonian decomposition of $K_n$. This was later generalized by Nash-Williams \cite{Nash87}. Hilton and Rodger \cite{HR} considered the embedding problem for Hamiltonian decompositions of complete multipartite graphs. For embeddings of factorizations  in which connectivity is also addressed, see \cite{HJRW,  MatJohns, RW}.

It is worth remarking that embeddings of combinatorial structures with the same flavor as results found in this paper have  a long history. For example, 
in his 1945 paper \cite{MHall45},  Hall proved that every $p\times n$ latin rectangle on $n$ symbols can be embedded in a latin square of size $n$. Following this classic embedding theorem, in 1951 Ryser generalized Hall's result to $p\times q$ latin rectangles on $n$ symbols \cite{HJRyster51}. Ryser's result is equivalent to embedding a proper edge-coloring of the complete bipartite graph $K_{p,q}$ into a $1$-factorization of $K_{n,n}$. 
Doyen and Wilson \cite{DoyWilson73} solved the embedding problem for  Steiner triple systems ($K_3$-decompostions of $K_n$), then Bryant and Horsley \cite{BryHors09} addressed the embedding of partial designs, proving Lindner's conjecture \cite{Lindner75} that any partial Steiner triple system of order $u$, $PSTS(u)$,  can be embedded in an $STS(v)$ if $v \equiv 1,3 \pmod 6$ and $v \geq 2u + 1$. ($2u+1$ is best possible in the sense that for all $u\geq 9$ there exists a $PSTS(u)$ that can not be embedded in an $STS(v)$ for any $v< 2u+1$.)

 \section {Detachments of amalgamated hypergraphs} \label{prelimd}
For the purpose of this paper, a \textit {hypergraph} $\scr{G}$ is an ordered quintuple $(V(\scr{G}),  E(\scr{G}), H(\scr{G}),$ $\psi, \phi)$ where $V(\scr{G}),  E(\scr{G}),  H(\scr{G})$ are disjoint finite sets, $\psi:H(\scr{G}) \rightarrow V(\scr G)$ is a function and  $\phi: H(\scr G) \rightarrow E(\scr G)$ is a surjection.  
  Elements of $V(\scr{G}), E(\scr{G}), H(\scr{G})$ are called \textit{vertices}, \textit{hyperedges} and \textit{hinges} of $\scr G$, respectively. 
A vertex $v$ (hyperedge $e$, respectively) and hinge $h$ are said to be \textit{incident} with each other if $\psi(h)=v$ ($\phi(h)=e$, respectively). A hinge $h$ is said to \textit{attach} the hyperedge $\phi(h)$ to the vertex $\psi(h)$. In this manner, the vertex $\psi(h)$ and the hyperedge $\phi(h)$ are said to be \textit{incident} with each other. If $e\in E(\scr G)$, and $e$ is incident with $n$ hinges $h_1, \ldots, h_n$ for some $n\in \mathbb N$ ($\mathbb N$ is the set of positive integers), then the hyperedge $e$ is said to \textit{join} (not necessarily distinct) vertices $\psi(h_1), \ldots, \psi(h_n)$. 
 If $v\in V(\scr G)$, then the number of hinges incident with $v$ is  called the \textit{degree} of $v$ and is denoted by $d_{\scr G}(v)$.  The number of distinct vertices incident with a hyperedge $e$, denoted by $|e|$, is called the \textit{size} of $e$.
If for all hyperedges $e$ of $\scr G$, $|e|\leq 2$ and $|\phi^{-1}(e)|=2$, then $\scr G$ is a \textit{graph}.

Note that a hypergraph as defined here corresponds to a hypergraph as usually defined providing hyperedges are allowed to contain vertices multiple times. We imagine each hyperedge of a hypergraph to be attached to the vertices which it joins by in-between objects called hinges. 
A hypergraph may be drawn as a set of points representing the vertices. A hyperedge is represented by a simple closed curve enclosing its incident vertices. A hinge is represented by a small line attached to the vertex incident with it (see Figure \ref{figure:hypexample1}). Any undefined term may be found in \cite{Bahhyp1}. 

If $\scr F=(V,  E, H, \psi, \phi)$ is a hypergraph and $\Psi$ is a function from $V$ onto a set $W$, then we shall say that the hypergraph $\scr G=(W,  E, H, \Psi \circ \psi, \phi)$ is an \textit{amalgamation} of $\scr F$ and that $\scr F$ is a \textit{detachment} of $\scr G$. Associated with $\Psi$ ($\Psi$ is called the  {\it amalgamation function}) is the \textit{number function} $g:W\rightarrow \mathbb N$  defined by $g(w)=|\Psi^{-1}(w)|$, for each $w\in W$; being more specific, we may also say that $\scr F$ is a \textit{g-detachment} of $\scr G$. Intuitively speaking, a $g$-detachment of $\scr G$ is obtained by splitting each $u\in V(\scr G)$ into $g(u)$ vertices. 
 Thus $\scr F$ and $\scr G$ have the same hyperedges and hinges, and each vertex $v$ of $\scr G$ is obtained by identifying those vertices of $\scr F$ which belong to the set $\Psi^{-1}(v)$. In this process, a hinge incident with a vertex $u$ and a hyperedge $e$ in $\scr F$ becomes incident with the vertex $\Psi(u)$ and the  hyperedge $e$ in $\scr G$. Since two hypergraphs $\scr F$ and $\scr G$ related in the above manner have the same hyperedges, coloring the hyperedges of one of them is the same thing as coloring the hyperedges of the other. Hence an amalgamation of a hypergraph with colored hyperedges is a hypergraph with colored hyperedges. 

\begin{example}\label{hyp1ex}
\textup{
 Let $\scr F=(V,  E, H, \psi, \phi)$, with $V=\{v_i: 1\leq i\leq 8\}, E=\{e_1,e_2, e_3\}, H=\{h_i: 1\leq i\leq 9\}$, such that for $1\leq i\leq 8$, $\psi(h_i)=v_i$ , $\psi(h_9)=v_6$ and $\phi(h_1)= \phi(h_2)=\phi(h_3)=e_1,\phi(h_4)=\phi(h_5)=\phi(h_6)=e_2, \phi(h_7)=\phi(h_8)=\phi(h_9)=e_3$. Let $\Psi:V\rightarrow \{w_1,w_2,w_3,w_4\}$ be the function with $\Psi(v_1)=\Psi(v_2)=\Psi(v_3)=w_1$, $\Psi(v_4)=w_2$, $\Psi(v_5)=\Psi(v_6)=w_3$, $\Psi(v_7)=\Psi(v_8)=w_4$. The hypergraph $\scr G$ is the $\Psi$-amalgamation of $\scr F$ (see Figure \ref{figure:hypexample1}). 
\begin{figure}[htbp]
\begin{center}
\scalebox{.7}{ \input {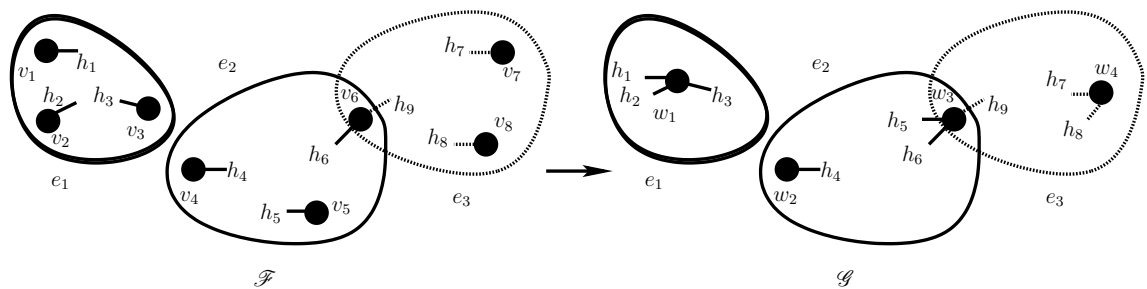} }
\caption{A visual representation of a hypergraph $\scr F$ with an amalgamation $\scr G$  }
\label{figure:hypexample1}
\end{center}
\end{figure} 
}\end{example}
A hypergraph $\scr G$ is said to be \textit{$r$-regular} if every vertex has degree $r$. An \textit{$r$-factor} of $\scr G$ is an $r$-regular spanning sub-hypergraph of $\scr G$. An \textit{$r$-factorization} is a set of $r$-factors whose edge sets partition $E(\scr G)$. A hypergraph $\scr G$ is said to be \textit{$k$-uniform} if $|e|=|\phi^{-1}(e)|=k$ for each $e\in E(\scr G)$ (so each vertex in a hyperedge is attached to it by exactly one hinge). A $k$-uniform hypergraph with $n$ vertices is said to be \textit{complete}, denoted by $K_n^k$, if every $k$ distinct vertices are joined exactly by one hyperedge. 
The sub-hypergraph of $\scr G$ induced by the color class $j$ is denoted by $\scr G(j)$. 

In the remainder of this paper, all hypergraphs are either 3-uniform or are amalgamations of 3-uniform hypergraphs. This implies that for every hypergraph $\scr G$ we have 
\begin{equation} \label{edgeassump}
1\leq |e|\leq |\phi^{-1}(e)|=3  \mbox{ for every  } e \mbox{ in } \scr G. 
\end{equation}
If $u,v,w$ are three (not necessarily distinct) vertices of $\scr G$, then $m(u,v,w)$ denotes the number of  hyperedges that join  $u$, $v$, and $w$. For convenience, we let  $m(u^2, v)=m(u, u, v)$, and  $m(u^3)=m(u, u, u)$. If we think of an edge as a multiset, then $m(u^2, v)$ (or $m(u^3)$) counts the multiplicity of an edge of the form $\{u,u,v\}$ (or $\{u,u,u\}$, respectively).

For the purpose of this paper, we need the following result which is a special case of both Theorem {\bf 3.1} in \cite{Bahhyp1}, and Theorem {\bf 4.1} in \cite{Bahhyp2bergJ}. To state it, some notation must be introduced. 
 
 For $g:V(\scr F)\rightarrow \mathbb{N}$, we define the symmetric function $\tilde g:V^3(\scr F)\rightarrow \mathbb{N}$ such that for distinct $x,y,z\in V(\scr F)$,  $\tilde g(x,x,x)=\binom{g(x)}{3}$, $\tilde g(x,x,y)=\binom{g(x)}{2}g(y)$, and $\tilde g(x,y,z) =g(x)g(y)g(z)$. 
Also we assume that for each $x \in V(\scr F)$, $g(x) \leq 2$ implies $m_\scr{F} (x^3) = 0$, and $g(x)=1$ implies $m_\scr F (x^2,y)=0$ for every $y\in V(\scr F)$. If $x, y$ are real numbers, then $\lfloor x \rfloor$ and $\lceil x \rceil$ denote the integers such that $x-1<\lfloor x \rfloor \leq x \leq \lceil x \rceil < x+1$, and $x\approx y$ means $\lfloor y \rfloor \leq x\leq \lceil y \rceil$. 
\begin{theorem} \textup{(Bahmanian \cite[Theorem 3.1]{Bahhyp1})} \label {hyp1main}
Let $\scr F$ be a $k$-hyperedge-colored hypergraph and let $g$ be a function from $V(\scr F)$ into $\mathbb{N}$. Then there exists a $3$-uniform $g$-detachment $\scr G$ of $\scr F$ with amalgamation function $\Psi:V(\scr G)\rightarrow V(\scr F)$,  $g$  being the number function associated with $\Psi$, such that:
\begin{itemize}
\item [\textup{(A1)}] for each $x\in V(\scr F)$, each $u\in \Psi^{-1}(x)$ and each $j\in \{1,\ldots,k\}$
$$d_{\scr G(j)}(u) \approx \frac{d_{\scr F(j)}(x)}{g(x)}; \mbox{ and }$$
\item [\textup{(A2)}] for every $x,y,z\in V(\scr F)$,  with $g(x)\geq 3$ if $x=y=z$, and $g(x)\geq 2$ if $|\{x,y,z\}|=2$, and every triple of distinct vertices $u,v,w$ with  $u\in \Psi^{-1}(x)$, $v\in \Psi^{-1}(y)$ and $w\in \Psi^{-1}(z)$, 
$$m_\scr G(u, v, w) \approx \frac{m_\scr F(x,y,z)}{\tilde g(x,y,z)}. $$  
\end{itemize}
\end{theorem}

\section{Embedding partial hyperedge-colorings into factorizations}\label{embedfactorizationcor1}

In this section we completely solve the embedding problem in the case where all the hyperedges of size 3 on a set of $m$ vertices have been colored, providing $n$ is big enough. We then show that some lower bound on $n$ is needed, since the necessary conditions of Theorem \ref{facembedgen1} are not sufficient if $n = 2m - 1$.  
\begin{theorem} \label{facembedgen1}
Suppose that  $n> 2m+(-1+\sqrt{8m^2-16m-7})/2$. A $q$-hyperedge-coloring of $\scr F=K_m^3$ can be embedded into an $r$-factorization of $\scr G=K_{n}^3$ if and only if 
\begin{itemize}
\item [\textup {(i)}] $3\divides rn,$
\item [\textup {(ii)}]  $r \divides \binom{n-1}{2},$ 
\item [\textup {(iii)}] $q\leq\binom{n-1}{2}/r,$ and 
\item [\textup {(iv)}] $d_{\scr F(j)}(v)\leq r$ for each $v\in V(\scr F)$ and $1\leq j \leq q.$ 
\end{itemize}
\end{theorem}
\begin{proof}
To prove the necessity, suppose that $\scr F$ with $V=V(\scr F)$ can be embedded into an $r$-factorization of $\scr G$. Since each edge contributes 3 to the the sum of the degrees of the vertices in an $r$-factor, $r | V(\scr G)|$ must be divisible by 3 which implies (i). Since each $r$-factor is an $r$-regular spanning sub-hypergraph and $\scr G$ is $\binom{n-1}{2}$-regular, we must have $r \divides \binom{n-1}{2}$, which is condition (ii). This $r$-factorization   requires  exactly $k=\binom{n-1}{2}/r$ colors which is condition (iii), and to be able to extend each color class to an $r$-factor  we need condition  (iv).

Now assume that conditions (i)--(iv) are true. By Baranyai's theorem \cite {Baran75}, the case of $m\leq 3$ is trivial, and so we may assume that $m\geq 4$. Let $e_j=|E\big(\scr F(j)\big)|$ for $1\leq j\leq k$.  In what follows, we extend the hyperedge-coloring of $\scr F$ into a $k$-hyperedge-coloring of an amalgamation of $\scr G$, and then we apply Theorem \ref{hyp1main} to obtain the detachment $\scr G$ in which each color class is an $r$-factor.  The hyperedges added in steps (I), (II), and (III)  correspond to the  hyperedges in $\mathscr G$ that contain one, two, and three new vertices, respectively. 
\begin{enumerate}[(I)]
\item Let $\scr F_1$ be a hypergraph formed by adding a new vertex $u$ and hyperedges  to $\scr F$ such that $m(u,v,w)=n-m$ for every pair of distinct vertices $v,w\in V$. Of course  the hyperedges in $E(\scr F)\cap E(\scr F_1)$ are already colored.  
We color greedily as many of the added $(n-m)\binom{m}{2}$ hyperedges as possible, ensuring  that $d_{\scr F_1(j)}(v)\leq r$ for $1\leq j\leq k$. 
Suppose there exists a hyperedge incident with $u$,$v$ and $w$ that is not colored. Then for $1\leq j\leq k$ either  $d_{\scr F_1(j)}(v)=r$ or $d_{\scr F_1(j)}(w)=r$, so $d_{\scr F_1(j)}(v)+d_{\scr F_1(j)}(w) \geq r$ for every $1\leq j\leq k$.
Therefore $2\binom{m-1}{2}+2(n-m)(m-1)-2=d_{\scr F_1}(v)+d_{\scr F_1}(w)-2 \geq \sum_{j=1}^k \big(d_{\scr F_1(j)}(v)+d_{\scr F_1(j)}(w)\big) \geq \sum_{j=1}^k r = kr=\binom{n-1}{2}$, in which the first inequality follows from that fact that at least one hyperedge incident with $v$ and $w$ is not colored. 
So, $2(m-1)(m-2)+4(n-m)(m-1)-4\geq (n-1)(n-2)$. Thus $n^2-4nm+n+2m^2+2m+2\leq  0$. So 
$$n\leq 2m+(-1+\sqrt{8m^2-16m-7})/2,$$
a contradiction. So all hyperedges can be colored greedily. Let $f_j$ be the number of  hyperedges of color $j$ in some such coloring for $1\leq j\leq k$.

\item Let $\scr F_2$ be a hypergraph formed by adding  $m\binom{n-m}{2}$ further hyperedges to $\scr F_1$ so that $m(u^2,v)=\binom{n-m}{2}$ for each $v\in V$.  Note that for each $v\in V$, 
\begin{eqnarray*} \label{rkdegverif}
d_{\scr F_2}(v)&=&\binom{m-1}{2}+ (m-1)(n-m)+\binom{n-m}{2} \nonumber \\
& =&\binom{n-1}{2}=rk.
\end{eqnarray*} 
Since $d_{\scr F_1(j)}(v)\leq r$ for $v\in V$ and $1\leq j\leq k$, to ensure  that $d_{\scr F_2(j)}(v)=r$, we color $r-d_{\scr F_1(j)}(v) (\geq 0)$  hyperedges incident with $v$ that were  added in forming $\scr F_2$ from $\scr F_1$ with color $j$ for each $v\in V$ and $1\leq j\leq k$. So the coloring we perform in this step results in all the newly added hyperedges being colored.  
Let $g_j$ denote the number of such hyperedges of color $j$ for $1\leq j\leq k$. 

\item Let $\scr F_3$ be the hypergraph formed by adding  $\binom{n-m}{3}$ further hyperedges to $\scr F_2$ so that $m(u^3)=\binom{n-m}{3}$. 
Let $\ell_j:=r(n/3-m)+f_j+2e_j$ for $1\leq j\leq k$. We claim that  $\ell_j\geq 0$ for $1\leq j\leq k$. To prove this, it is enough to show that $n\geq 3m$. 
Since $m\geq 4>(3+\sqrt{17})/2$, we have $m^2-3m-2\geq 0$. Therefore, $8m^2-16m-7\geq 4m^2-4m+1$, and thus $\sqrt{8m^2-16m-7}\geq 2m-1$, which implies   $(1+\sqrt{8m^2-16m-7})/2\geq m$, and consequently we have   $\lfloor (1+\sqrt{8m^2-16m-7})/2\rfloor \geq m$. Since $n> 2m+\lfloor(1+\sqrt{8m^2-16m-7})/2\rfloor$, we have $n\geq 3m$. 

Now we color the added hyperedges such that there are exactly $\ell_j$ further hyperedges colored $j$  for $1\leq j\leq k$. This is possible because
\begin{eqnarray*}
\sum_{j=1}^k \ell_j&=&\sum_{j=1}^k \big(r(n/3-m)+f_j+2e_j\big)\\
&= & rk(n/3-m)+\sum_{j=1}^k f_j+2\sum_{j=1}^k e_j\\
&= & \binom{n-1}{2}(n/3-m)+(n-m)\binom{m}{2}+2\binom{m}{3}\\
&=& n^3/6-n^2m/2-n^2/2+nm^2/2+nm\\
& & +\  n/3-m^3/6-m^2/2-m/3 \\
&= &  \binom{n-m}{3}=m_{\scr F_3}(u^3).
\end{eqnarray*}
Let us fix  $j\in \{1,\dots, k\}$. Since $d_{\scr F_3(j)}(v)=r$ for $v\in V$, we have 
\begin{equation} \label{degcalc12}
rm=\sum_{v\in V} d_{\scr F_3(j)}(v)=3e_j+2f_j+g_j.
\end{equation} 
On the other hand,   
\begin{eqnarray*}
d_{\scr F_3(j)}(u) & = & 3\ell_j+2g_j+f_j=r(n-3m)+3f_j+6e_j+2g_j+f_j \\
&=& r(n-3m)+4f_j+6e_j+2g_j.
\end{eqnarray*}
This together with (\ref{degcalc12}) implies that for $1\leq j\leq k$, 
$$d_{\scr F_3(j)}(u)=r(n-3m)+2rm=r(n-m).$$

\item Let $g:V(\scr F_3)\rightarrow \mathbb N$ be a function with $g(u)=n-m$, and $g(v)=1$ for each $v\in V$. By Theorem \ref{hyp1main}, there exists a 3-uniform $g$-detachment $\scr G^*$ of $\scr F_3$ with $n-m$ new vertices, say $u_1,\ldots, u_{n-m}$ detached from $u$ such that 
\begin{itemize}
\item $d_{\scr G^*(j)}(v)=d_{ \scr F_3(j)}(v)/g(v)=r/1=r$ and $d_{\scr G^*(j)}(u_i)=d_{ \scr F_3(j)}(u)/g(u)=r(n-m)/(n-m)=r$  for $1\leq i\leq n-m$ and $1\leq j\leq k$; 
\item $m_{\scr G^*}(u_i, u_{i'}, u_{i''})=m_{\scr F_3}(u^3)/\binom{g(u)}{3}=\binom{n-m}{3}/\binom{n-m}{3}=1$ for $1\leq i<i'<i''\leq n-m$; 
\item  $m_{\scr G^*}(u_i, u_{i'}, v)=m_{\scr F_3}(u^2,v)/\big(\binom{g(u)}{2}g(v)\big)=\binom{n-m}{2}/\binom{n-m}{2}=1$ for $1\leq i<i'\leq n-m$, and $v\in V$, and 
\item  $m_{\scr G^*}(u_i, v, w)=m_{\scr F_3}(u,v,w)/\big(g(u)g(v)g(w)\big)=(n-m)/(n-m)=1$ for $1\leq i\leq n-m$ and distinct  $v,w\in V$.
\end{itemize}
Therefore $\scr G^*\cong \scr G=K_n^3$ and each color class is an $r$-factor. This  completes the proof.
\end{enumerate}
\end{proof}
\begin{lemma} \label{lowerboundmn}
Conditions \textup {(i)--(iv)} of Theorem \ref{facembedgen1} are not sufficient if $n = 2m - 1$.
\end{lemma}
\begin{proof}
Suppose that the hyperedge-coloring of $K_m^3$ induces an $r$-factorization. 
Then in the embedding, the sub-hypergraph of $K_n^3$ on the new $n-m$ vertices induced by the hyperedges having the original colors clearly has an $r$-factorization (each of the colors induces an $r$-factor).  Therefore $n-m\geq m$, or equivalently $n\geq 2m$. So if $r$ is chosen so that $3 \divides r$ and $r\divides m -1$, then it is easy to check that conditions (i)--(iv) of Theorem \ref{facembedgen1} are satisfied when $n = 2m -1$, yet no embedding is possible.
\end{proof}

\section{extending restrictions of partial edge-colorings} \label{embedrestrfac}
If every hyperedge $e$ of the hypergraph $\scr G$ is replaced with $\lambda$ ($\geq 2$) copies of $e$ then denote the resulting (multi) hypergraph by $\lambda \scr G$. If $\scr G_1, \dots, \scr G_t$ are hypergraphs on the vertex set $V$ with edge sets $E(\scr G_1)\dots, E(\scr G_t)$ respectively, then let $\bigcup_{i=1}^t \scr G_i$ be the hypergraph with vertex set  $V$ and edge set $\bigcup_{i=1}^t E(\scr G_i)$. 

In this section we completely solve the embedding problem in the case where all the hyperedges in $\scr F=K_m^3\cup (n-m)K_m^2 \cup \binom{n-m}{2} K_m^1$ on a set of $m$ vertices have been colored, regardless of the size of $n$. One can think of the given colored hyperedges as being all the ``pieces" of hyperedges on these $m$ vertices that are eventually extended to hyperedges of size 3 by adding the new $n - m$ vertices during the embedding process.

Let $E^i(\scr G (j))$ denote the set of hyperedges of size $i$ and color $j$ in $\scr G$. 
\begin{theorem} \label{restricembedth1}
A $k$-hyperedge-coloring of $\scr F=K_m^3\cup (n-m)K_m^2 \cup \binom{n-m}{2} K_m^1$ with $V=V(\scr F)$ can be extended to an $r$-factorization of $\scr G=K_n^3$ if and only if
\item [\textup {(i)}] $3 \divides rn,$
\item [\textup {(ii)}]  $r \divides \binom{n-1}{2},$ 
\item [\textup {(iii)}] $k=\binom{n-1}{2}/r$,  
\item [\textup {(iv)}] $d_{\scr F(j)}(v)= r$ for each $v\in V$ and $1\leq j \leq k$, and 
\item [\textup {(v)}] $|E^2(\scr F(j))|+2|E^3(\scr F(j))|\geq r(m-n/3)$ for $1\leq j \leq k$.
\end{theorem}
\begin{proof} First, suppose that $\scr F$ can be embedded into an $r$-factorization of $\scr G$. The necessity of (i)--(iv) follow as described in the proof of Theorem \ref{facembedgen1}; equalities in this result replace the inequalities there because the colors of all hyperedges restricted to $\scr F$ have been prescribed in this case. Let us fix $j\in \{1,\dots,k\}$. Let $e_j, f_j, g_j$, and $\ell_j$ be the number of  hyperedges in $E(\scr G(j))$ that are incident with exactly 3, 2, 1 and 0 vertices  in $V$, respectively. It is easy to see that  $e_j=|E^3(\scr F(j))|$ and $f_j=|E^2(\scr F(j))|$. Since $r(n-m)=3\ell_j+2g_j+f_j$, and $rm=g_j+2f_j+3e_j$, we have $r(n-3m)=3\ell_j-3f_j-6e_j$, and thus $\ell_j=r(n/3-m)+f_j+2e_j$, but since $\ell_j\geq 0$, we must have $2e_j+f_j\geq r(m-n/3)$. This proves (v). 

To prove the sufficiency, assume that conditions (i)--(v) are true. Let $\scr F'$ be a hypergraph formed by adding a new  vertex $u$ to $\scr F$ with $m(u^3)=\binom{n-m}{3}$,  and extending each hyperedge of size one or two to a hyperedge incident with $u$ of size two or three, respectively. We extend the hyperedges of size one (two, respectively) such that $u$ is incident with two (one, respectively) hinges within that hyperedge. Ignoring colorings, $\scr F'$ is isomorphic to $\scr F_3$ in the proof of Theorem \ref{facembedgen1}, and  $\scr F'$ is an amalgamation of $\scr G$. We color $r(n/3-m)+f_j+2e_j$ of the new hyperedges with color $j$. This coloring results in all the newly added hyperedges being colored. 
The rest of the proof is identical to part (IV) of Theorem \ref{facembedgen1}. 
\end{proof}

\bibliographystyle{model1-num-names}
\bibliography{<your-bib-database>}

\end{document}